\date{}
\title{One-ended spanning trees in amenable unimodular graphs}

\author{\'Ad\'am Tim\'ar\\ 
\small Alfr\'ed R\'enyi Institute of Mathematics\\[-0.8ex]
\small Re\'altanoda u. 13-15.,\\
\small H-1053 Budapest\\
\small \texttt{madaramit[at]gmail.com}\\}

\documentclass[10pt,naturalnames]{article}
\renewcommand\footnotemark{}
\usepackage[ruled,vlined]{algorithm2e}
\usepackage{amsmath}
\usepackage{amsthm}
\usepackage{amsfonts}
\usepackage{graphicx}
\newif\ifhyper\IfFileExists{hyperref.sty}{\hypertrue}{\hyperfalse}
\ifhyper\usepackage{hyperref}\fi
\usepackage{enumitem}
\usepackage{subfig}
\usepackage{caption}
\usepackage{keyval}
\usepackage[margin=1in]{geometry}
\usepackage{setspace}
\usepackage[displaymath, mathlines,pagewise]{lineno}

\newif\ifdraft

\theoremstyle{definition}

\newtheorem{theorem}{Theorem}
\newtheorem{corollary}[theorem]{Corollary}
\newtheorem{lemma}[theorem]{Lemma}

\newcommand{\Z}{\mathbb{Z}}

\def \eps {\epsilon}

\def \P {{\Bbb P}}

\def \E {{\Bbb E}}

\def \_reg {\rightarrow_{\bf reg}}

\def\maxdeg/{\Delta}

\def\dist{{\rm dist}}

\def\dist{{\rm dist}}

\def \eps {\epsilon}

\def \P {{\bf P}}

\def \E {{\bf E}}

\def \_reg {\rightarrow_{\bf reg}}

\def\maxdeg/{\Delta}

\def\dist{{\rm dist}}

\def\cals{{\cal S}}
\def\cale{{\cal E}}
\def\Pcal{{\cal P}}

\def\dist{{\rm dist}}

\def\cali{{\cal I}}

\begin{document}
\maketitle
\let\thefootnote\relax\footnotetext{\footnotesize{This research was supported by the Hungarian National Research, Development and Innovation Office, NKFIH grant K109684, and by grant LP 2016-5 of the Hungarian Academy of Sciences.}}

\bigskip

\begin{abstract}
We prove that every amenable one-ended Cayley graph has an invariant spanning tree of one end. More generally, for any 1-ended amenable unimodular random graph we construct a factor of iid percolation (jointly unimodular subgraph) that is almost surely a spanning tree of one end. In \cite{BLPS} and \cite{AL} similar claims were proved, but the resulting spanning tree had 1 or 2 ends, and one had no control of which of these two options would be the case. 
\end{abstract}

\vspace{0.5in}

Every unimodular amenable graph $G$ allows a percolation (random subgraph whose distribution is jointly unimodular with $G$) that is almost surely a spanning tree with 1 or 2 ends; see Theorem 8.9 in \cite{AL}. We strengthen this by showing that if $G$ is amenable and 1-ended then it has a 1-ended spanning tree percolation, and this can be constructed as a factor of iid (fiid). This later condition for fiid construction was already implicit in \cite{AL} (and \cite{BLPS}), so the real novelty is that we do not have to allow 2-ended trees. Our original motivation was \cite{T2}, where it was crucial that the spanning forest is an fiid and has 1 end. 
See Section 8 of \cite{AL} for the generalized definition of amenability to unimodular random graphs and for several equivalent characterizations.  

\begin{theorem}\label{oneended}
Let $G$ be an ergodic amenable unimodular random graph that has one end almost surely. Then there is a factor of iid spanning tree of $G$ that has one end almost surely.
\end{theorem}

In \cite{BLPS}, Benjamini, Lyons, Peres and Schramm proved that a quasi-transitive unimodular graph is amenable if and only if it has an invariant spanning tree with at most 2 ends (Theorem 5.3). Note that a quasi-transitive amenable graph can only have 1 or 2 ends; and also that if it has 2 ends then all its invariant spanning trees are 2-ended. Our result can hence be thought of as a strengthening of the characterization in \cite{BLPS}:
\begin{corollary}
A quasi-transitive unimodular graph is amenable and has 1 end if and only if it has an invariant spanning tree with 1 end.
\end{corollary}

We assume that the reader is familiar with the Mass Transport Principle (MTP). See e.g. \cite{AL} for the formulation, which can be taken as the defining property of unimodular graphs. A subgraph $H$ of the rooted graph $(G,o)$ is a factor of iid (fiid), if
it can be constructed as a Borel measurable function from iid Lebesgue($[0,1]$) labellings of $V(G)$ that is equivariant with rooted isomorphisms; in other words, if one can tell the edges of $H$ incident to $o$ up to arbitrary precision from the labels in a large enough neighborhood of $o$.
See e.g. \cite{T2} for a more formal definition.
Along the proofs we will make some local choices, such as choosing a subgraph of a certain property out of finitely many possibilities, otherwise arbitrarily. To make the final result a fiid, these choices have to be made using some previously fixed local rule using the iid labels. We will skip the details of such choices, which are straightforward.


\begin{lemma}\label{connectedsub}
Let $G$ be an ergodic amenable unimodular random graph. Suppose that there exists a factor of iid sequence $(H_n)$ of connected subgraphs of $G$ such that $\P (o\in H_n)\to 0$. Then $G$ has a 1-ended factor of iid spanning tree. 
\end{lemma}

\begin{proof}
By switching to a subsequence if necessary, we may assume that $\P (o\in H_n)< 2^{-n}$. 
We may also assume that $H_{n+1}\subset H_n$, as we explain next. First note that for every $\eps>0$ one can modify every $H_n$ to get an $H_n'$, in such a way that $H_n\subset H_n'$, $H_n'$ is invariant, connected, $\P (o\in H_n')< 2^{-n}(1+\eps)$, and $H'_n\cap H'_{n+1}\not=\emptyset$. Namely, suppose that the distance between $H_n$ and $H_{n+1}$ is $k$. If $k=0$, choose $H_n'=H_n$. Otherwise, for every point of $H_n$ at distance $k$ from $H_{n+1}$, fix a path of length $k$ between this point and $H_{n+1}$, and select it with probability $\eps/(k+1)$. Add all the selected paths to $H_n$ to obtain $H_n'$. There must be infinitely many points in $H_n$ at distance $k$ from $H_{n+1}$ by the MTP, so $H_n'$ in fact intersects $H_{n+1}\subset H_{n+1}'$ almost surely.
Therefore $H_n'':=\cup_{i=n}^\infty H_n'$ is connected, $H_{n+1}''\subset H_n''$, and $\P (o\in H_n'')< 2^{-n+1}(1+\eps)$, as we wanted. So we will assume that $H_{n+1}\subset H_n$.

Let $(\Pcal_m)$ be a sequence of partitions of $V(G)$ such that
every partition class induces a connected subgraph of $G$, $\Pcal_n$ is coarser than $\Pcal_{n-1}$, and any two points of $V(G)$ are in the same class of $\Pcal_n$ for all but finitely many $n$.
Such a sequence exists, see e.g. Theorem 5.3 in \cite{AL}, where such a sequence is used for unimodular graphs to construct a spanning tree of at most two ends. This construction is also fiid, which is implicit in the proof. 

As usual, let $o$ be the root of our unimodular graph. Denote by $x$ a uniformly chosen neighbor of $o$. If $H$ is a subgraph of $G$, $v$ and $w$ two vertices, we let $v\leftrightarrow_H w$ stand for the event that $v$ and $w$ are in the same component of $H$. For an arbitrary forest ${\cal F}$ and vertex $v$, let ${\cal F}(v)$ be the component of $v$ in ${\cal F}$.

We will define spanning forests $F_n$ of $G$ that all have only finite components, $(G,F_n)$ is jointly unimodular, and their limit will be the tree in the claim. Let $H_0:=G$ and $F_0:=\emptyset$.

Let $k(n)$ be a strictly increasing sequence of positive integers, to be defined later, with $k(0)=0$. Suppose recursively that $F_n$ has been defined, all its edges are in $G\setminus H_{k(n)}$, and every component of it is adjacent to $H_{k(n+1)}$. Suppose further that 
\begin{equation}
\P(x\leftrightarrow_{F_n} o)\geq 1-2^{-n}. 
\label{eq:recursive}
\end{equation}
The recursive assumptions trivially hold for $n=0$.

Figure \ref{reszletes} illustrates the steps of the construction that are explained next.

\begin{figure} 
\centering
\begin{minipage}{\columnwidth}
\centering
\subfloat[][attaching $F_n$ to $H_{k(n)}$]{\includegraphics[scale=0.45]{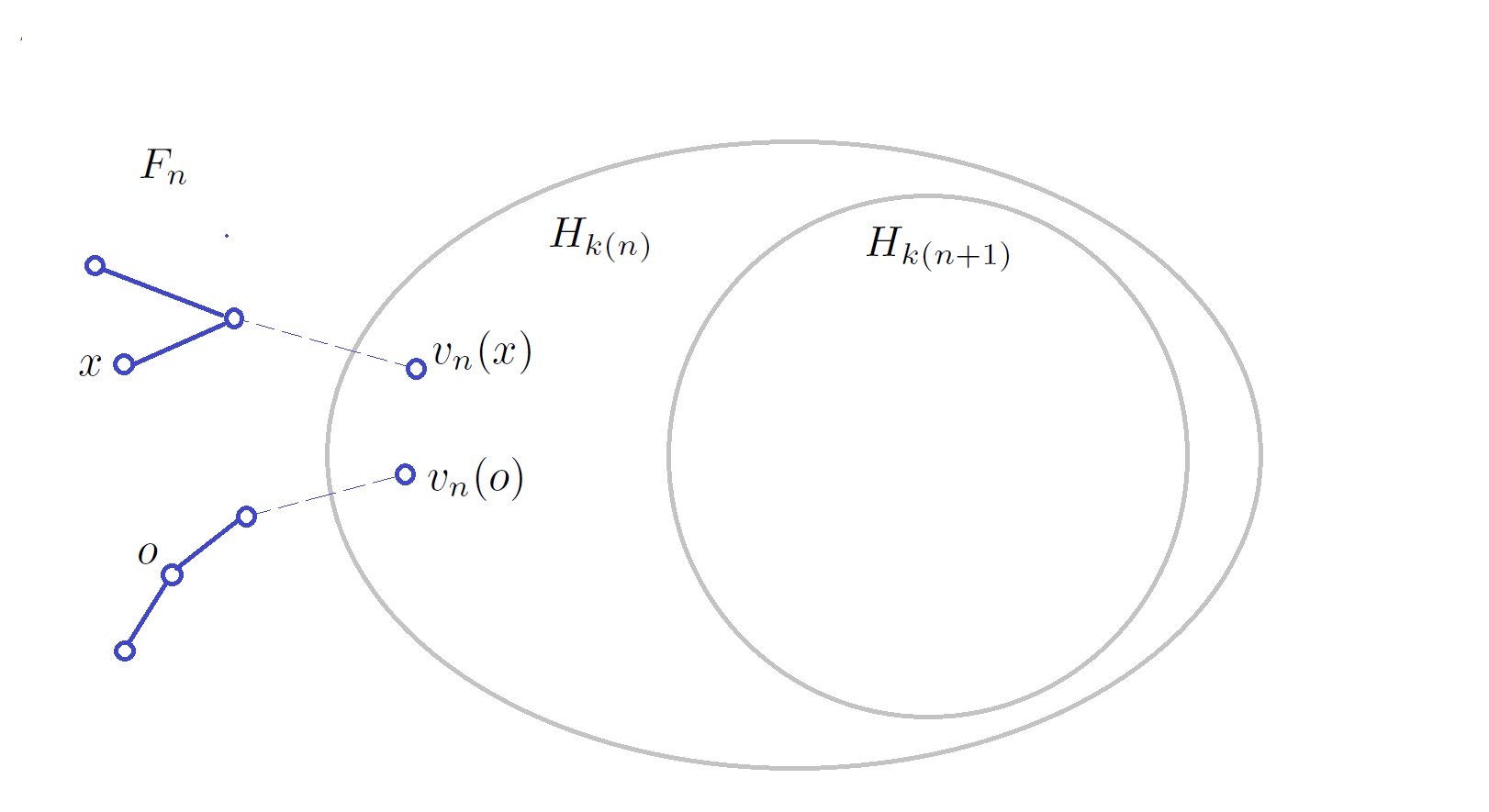} \qquad \includegraphics[scale=0.45]{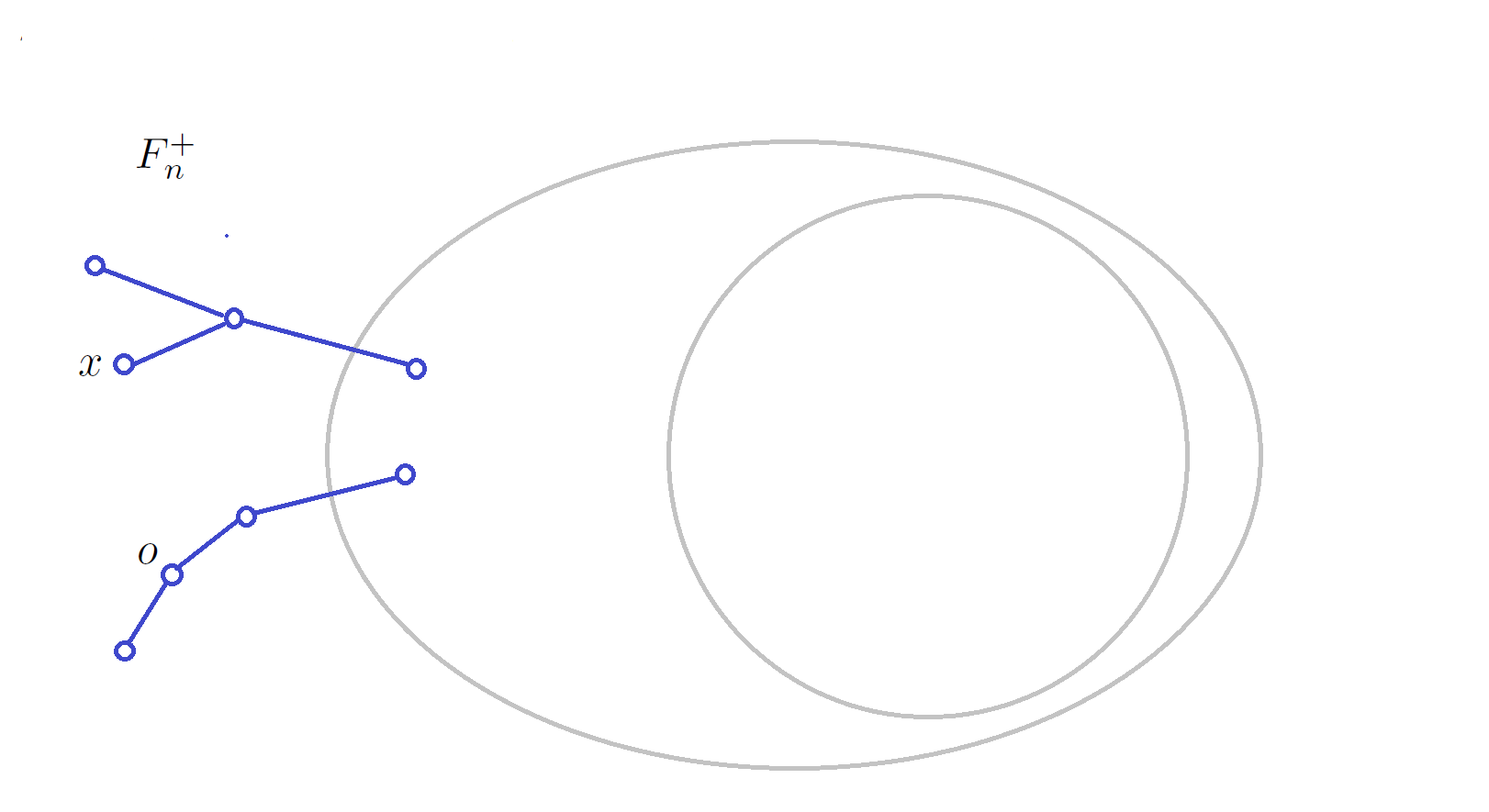}}
\end{minipage}
\begin{minipage}{\columnwidth}
\centering
\subfloat[][constructing the forest within $H_{k(n)}\setminus H_{k(n+1)}$]{\includegraphics[scale=0.45]{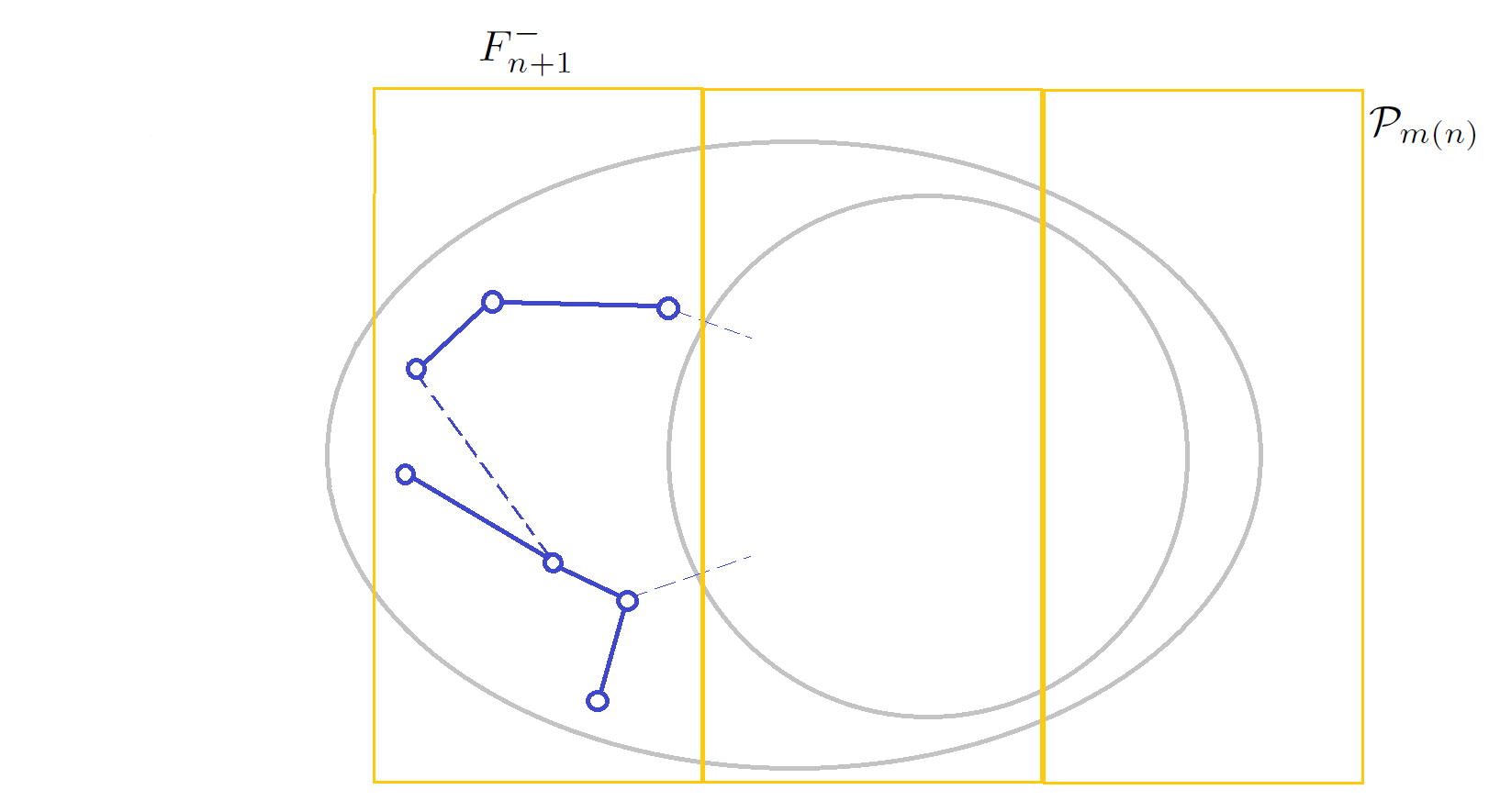} \qquad \includegraphics[scale=0.45]{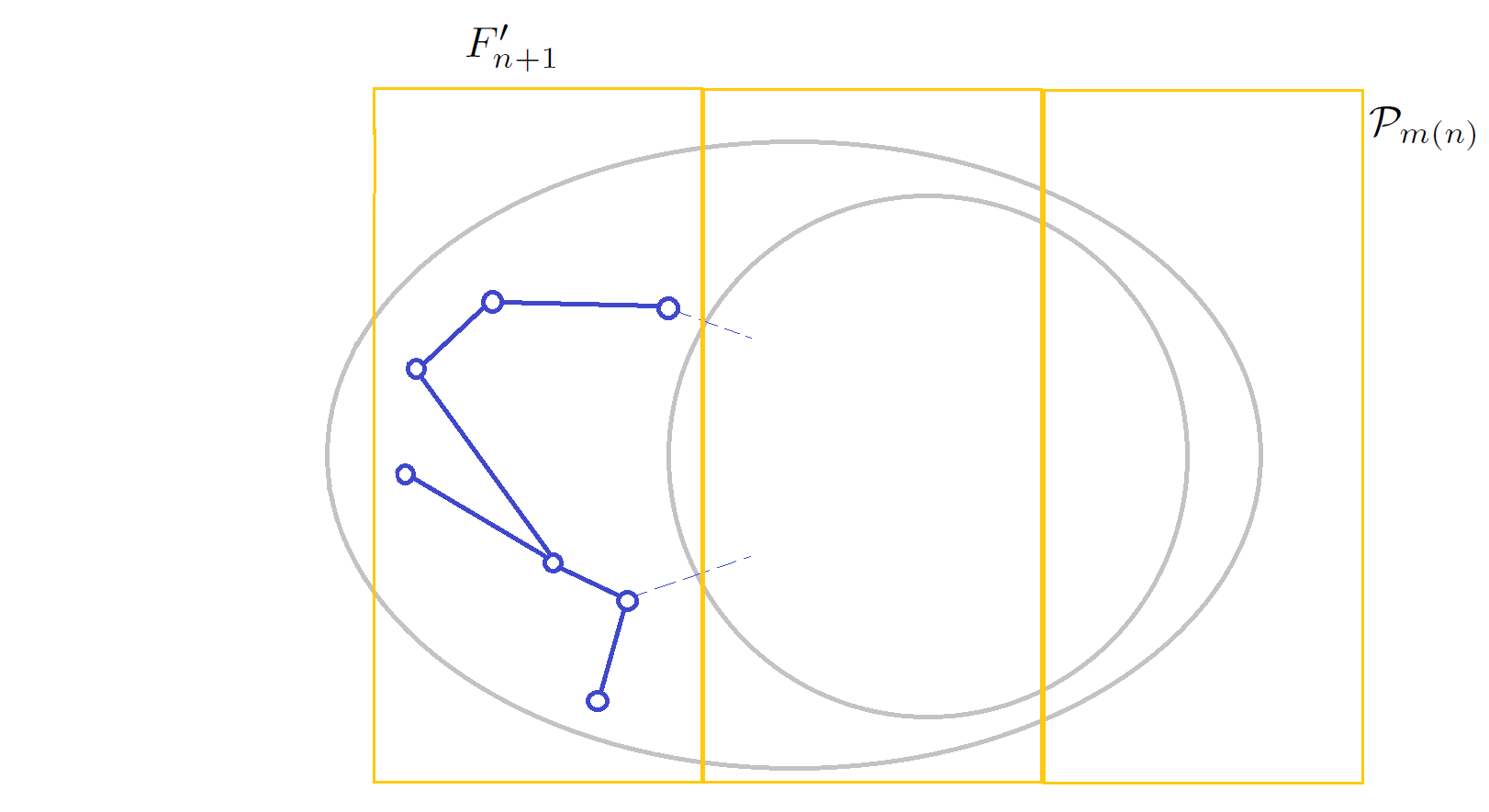}}
\end{minipage}
\begin{minipage}{\columnwidth}
\centering
\subfloat[][the union of the two provides $F_{n+1}$]{\includegraphics[scale=0.45]{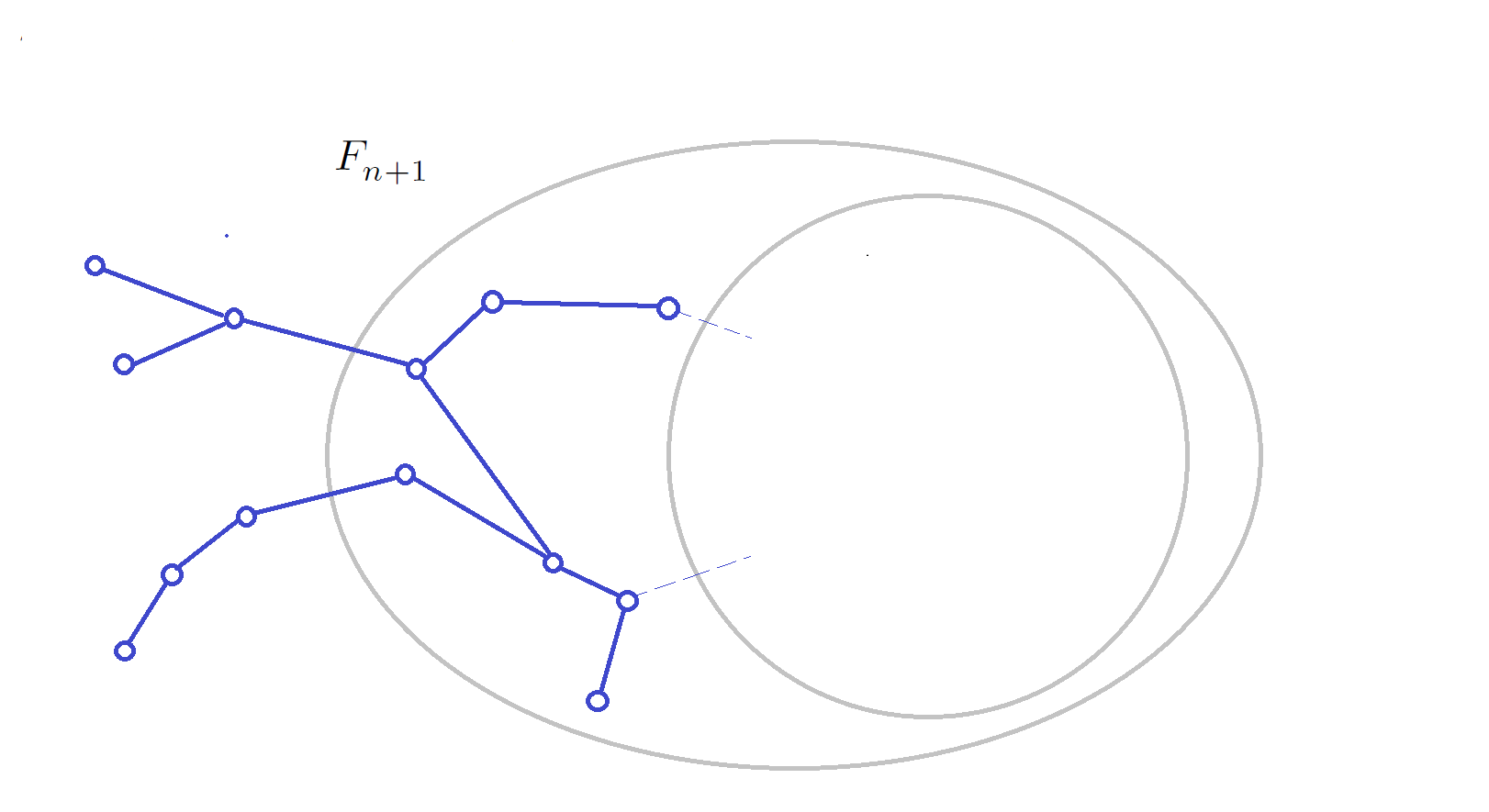}}
\end{minipage}
\caption{The construction of $F_{n+1}$ from $F_n$. Dashed lines are in $E(G)$, but not in the graph at display.}
\label{reszletes}
\end{figure}

For every component $C$ of $F_n$, let $v(C)$ be a randomly chosen vertex of $H_{k(n+1)}$ that is adjacent to $C$. Define $F_n^+$ as the union of $F_n$ and all the edges of the form $\{v(C), u\}$, where $u\in C$ and $C$ is a
component of $F_n$.
Let $v_n(x)$ (respectively $v_n(o)$) be equal to $v(C_x)$ (resp. $v(C_o)$), where $C_x$ is the component of $x$ (resp $o$) in $F_n$.

\def\up{{\partial^{{\rm up}}H_{k(n)}}}

Let $\up$ be the set of vertices in $H_{k(n)}\setminus H_{k(n+1)}$ that are adjacent to $H_{k(n+1)}$. 
Grow a forest within 
$H_{k(n)}\setminus H_{k(n+1)}$ starting from $\up$ iteratively as follows. As $i=0,1,\ldots$, consider the set $U_i$ of vertices at distance $i$ from $\up$ (so $U_0=\up$), and for $i\geq 1$ pick a randomly chosen edge between each vertex in $U_i$ and some vertex in $U_{i-1}$. As $i\to\infty$, we end up with a forest $F_{n+1}^-$ in $H_{k(n)}\setminus H_{k(n+1)}$, which has the property that each of its components contains a unique point of $\up$ (by the connectedness of $H_{k(n)}$), and consequently, each component if finite (by the MTP). 

Let $A(n,m)$ be the event that there is a path with consecutive vertices $ p_1,\ldots, p_\ell$, between $v_n(x)$ and $v_n(o)$ ($p_1=v_n(x), p_\ell=v_n(o)$), with the
$F_{n+1}^-(p_i)$ all fully contained in the same class of $\Pcal_m$.
By definition of $\Pcal_m$, we have that $\lim_{m\to\infty}\P(v_n(x)\leftrightarrow_{K_{k(n)}\setminus K_{k(n+1)}}v_n(o); A(n,m))
=\P(v_n(x)\leftrightarrow_{K_{k(n)}\setminus K_{k(n+1)}}v_n(o))$. Choose $m(n)$ large enough so that 
\begin{equation}
\P(v_n(x)\leftrightarrow_{K_{k(n)}\setminus K_{k(n+1)}}v_n(o); A(n,m))\geq1-2^{-n+1}.
\label{eq:szimes}
\end{equation}
Such a choice is possible by the recursive assumption \eqref{eq:recursive}.
For each class $K$ of $\Pcal_{m(n)}$ consider the set of components of $F_{n+1}^-$ that lie entirely in $K$, and add a maximal number of edges to them (following some otherwise arbitrary rule) so that the result is still cycle-free. Call the resulting forest $F_{n+1}'$ (so $F_{n+1}'$ is $F_{n+1}^-$ with all these added edges). Then, 
by \eqref{eq:szimes}, 
$\P(x\leftrightarrow_{F_{n+1}'\cup F_n^+}o)=
\P(v_n(x)\leftrightarrow_{F_{n+1}'}v_n(o))\geq 1-2^{-n+1}.$ Finally, define $F_{n+1}$ as $F_{n+1}'\cup F_n^+$. By construction, the recursive assumptions are satisfied by $F_{n+1}$.

Let $F$ be the limit of the increasing sequence $F_n$. It is clearly a forest, and by \eqref{eq:recursive}, $F$ is a spanning tree. To see that $F$ has one end, pick an arbitrary vertex $v$, and let $n\in\{0,1,\ldots\}$ be such that $v\in H_{k(n)}\setminus H_{k(n+1)}$. If $C$ is the component of $v$ in $F_n$, then $v$ is in a finite component of $F\setminus \{v(C)\}$, hence $v$ is separated from infinity by one point, as we wanted.
\end{proof}
\qed

In what follows we are going to construct a sequence of fiid connected subgraphs of $H_n$ with marginals tending to 0, as in Lemma \ref{connectedsub}. This will then establish Theorem \ref{oneended}.

From now on, {\it intervals always mean discrete intervals}, e.g. $[a,b]$ with $a,b\in\Z$ is the set $\{a,a+1,\ldots, b\}$. An interval may only consist of 1 point. Given a set of intervals, it will automatically define an {\it interval graph}, as the graph whose vertices are the given intervals, and two are adjacent if they intersect. By a slight sloppiness, we will refer to the graph induced by a set $\cali$ of intervals by the same notation $\cali$. 

\def\endpoints{{\cal V}}

\begin{lemma}\label{todelta}
Let $a,b\in\Z$, and let $\cali$ be a connected interval graph of intervals in $[a,b]$. Suppose that both $a$ and $b$ are contained in some interval in $\cali$. Then there is some $\cali'\subset\cali$ such that the graph induced by $\cali'$ is connected, and every integer of $[a,b]$ is contained in exactly 1 or 2 elements of $\cali'$. 

Denote the minimal length of an interval in $\cali$ by $\Delta$. Fix $\delta\leq\lfloor\Delta/2\rfloor$ to be a positive integer.
Define $O=\delta\Z$. Then there is a map $\iota$ from the set of endpoints $\endpoints(\cali)$ of $\cali$ to $O$ that has the following properties:
\begin{enumerate}
\item $|x-\iota (x)|\leq 2\delta$ for every $x\in\endpoints(\cali)$.
\item If $x\leq y$, $x,y\in \endpoints(\cali)$, then $\iota (x)\leq \iota(y)$. In particular, the interval graph defined by $\cali '':=\{[\iota (a),\iota (b)]:\, [a,b]\in\cali'
\}$ is such that $\iota$ maps adjacent (intersecting) intervals in $\cali'$ to adjacent intervals in $\cali''$.
\item Every point of $[a,b]$ is contained in at most two elements of $\cali''$.
\end{enumerate}
\end{lemma}

\begin{proof}
Choose a path $I_0,\ldots, I_m$ (with $I_i\cap I_{i+1}\not=\emptyset$)
in the interval graph $\cali$ with the property that $a\in I_0$, $b\in I_m$ (this latter we refer to by saying that the path bridges $a$ and $b$), and make the choice so that $m$ is minimal. By assumption, every point $k\in[a,b]$ is contained in some $I_i$. Suppose now that for some $k\in[a,b]$ there exist three distinct intervals that contain $k$. It is easy to check that then one can choose two of these three such that their union contains the third one. But then this third one could be dropped from $I_0,\ldots, I_b$, and one would still be left with a connected graph (and a path that bridges $a$ and $b$ in it), contradicting the minimality of $m$. Hence $\cali':=\{I_0,\ldots,I_m\}$ satisfies the first assertion. See Figure \ref{intervals} for the pattern of the intervals and the naming introduced in the next paragraph.

Still using that the $\cali'$ we defined is a minimal path, one can check the following. Denote the endpoints of
$I_0$ by $x_0$ and $x_2$, with $x_0<x_2$. Denote the endpoints of $I_m$ by $x_{2m-1}$ and $x_{2m+1}$, where $x_{2m-1}< x_{2m+1}$.
Finally, for $0<k<m$, let the endpoints of
 $I_k$ be $x_{2k-1}$ and $x_{2k+2}$, where $x_{2k-1}<x_{2k+2}$. Then $x_{2k-1}\leq x_{2k}$, because $I_{k-1}$ intersects $I_k$, and the latter is closer to $b$ than the former. Similarly, for $k\geq 1$ we have $x_{2k}<x_{2k+1}$, because $I_{k-1}\cap I_{k+1}=\emptyset$ (by the assumption that every point of $[a,b]$ is in at most two of the intervals).

\begin{figure}[h]
\vspace{0.1in}
\begin{center}
\includegraphics[keepaspectratio,scale=1.4]{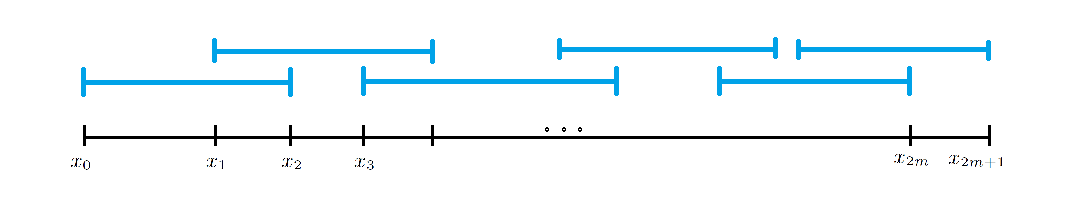}
\caption{Intervals representing a path in the interval graph. The lowest line represents the underlying set $[a,b]$.}\label{intervals}
\end{center}
\end{figure}

To construct $\iota$, do the following. 
Define a map $\iota'$ first, for $x\in\endpoints(\cali)$, by letting $\iota'(x)$ be the point of $O$ closest to $x$ (in case of a tie, decide arbitrarily). Vertex $x_{2k-1}$ and $x_{2k+2}$ are always at least $\Delta\geq 2\delta$ apart from each other, hence they cannot be mapped to the same point or to neighbors in $\delta\Z$. Therefore at most 3 points can be mapped to the same point by $\iota'$, and if
3 points are mapped to the same $v\in O$, then no point is mapped to $v+\delta$ or $v-\delta$.
Suppose that 3 vertices are mapped to some $v\in O$, that is, $\iota' (x_i)=\iota' (x_{i+1})=\iota' (x_{i+2})$. Then, if $x_{i+2}>\iota'(x_{i+2})$, define $\iota(x_{i+2})=\iota'(x_{i+2})+\delta$, and $\iota (x_i)=\iota (x_{i+1})=\iota' (x_{i+1})$. Otherwise we have $x_i<x_{i+1}<x_{i+2}\leq \iota'(x_{i+2})$. In this case define $\iota (x_i)=\iota'(x_i)-\delta$, and  $\iota (x_{i+1})=\iota (x_{i+2})=\iota' (x_{i+2})$. It is easy to check that $\iota$ satisfies the requirements.
\end{proof}
\qed

\begin{lemma}\label{distance}
Let $G$ and $B$ be random graphs, $B$ being a biinfinite path, and suppose that $(G,B)$ is jointly unimodular, $V(G)=V(B)$ and $E(B)\subset E(G)$. Suppose further that $G$ has only one end. Let $x_n$ and $x_{-n}$ be the two vertices whose distance from the root is $n$ in $B$.
Then
$$\lim_{n\to\infty}\E (\dist _G (x_{-n},x_n)/2n)\to 0.$$
\end{lemma}

\begin{proof}
There are two graph isomorphisms from $B$ to $\Z$ that take the root to 0, pick one of the two randomly with probability 1/2 and fix is, for simpler reference. Through this isomorphism, we can refer to the points of $B$ as integers; we will use $B$ and $\Z$ interchangeably. This way, to every edge $e=\{k,\ell\}\in E(G)$, we can assign the interval $I(e)=[k,\ell]$, which can be thought of as the unique path in $B$ between the endpoints of $e$. We refer to $\ell-k$ as the length of the edge $e$.
Note that by subadditivity the limit exists,
$\lim \E ( \frac{\dist _G(-n,n)}{2n})=\inf \E ( \frac{\dist _G(-n,n)}{2n}).$
We need to prove that this number is 0. 

Suppose to the contrary, that 
$\lim \E ( \frac{\dist _G(0,n)}{n})=\lim \E ( \frac{\dist _G(-n,n)}{2n})=c>0$. Let $c<c'<8c/7$. Let $d$ be a positive integer such that 
$\E(\dist_G (o, n))<c'n
)$ for every $n\geq d$. Pick some $D> (64d+32)/3c'$($>2d$).


\def\bfs{{\bf S}_0}

Because of the one-endedness of $G$, for any point $x\in \Z$, there are infinitely many intervals $I(e)$, $e\in E(G)$, that contain $x$. In other words, there are infinitely many edges whose endpoints belong to different components of $B\setminus\{x\}$. Hence we can choose some number $D'$ with the property that $\P(o\in I(e) \text{ for some }I(e) \text{ with } D\leq |I(e)|<D')>1-c'/32$. By unimodularity, we have the same probability if we replace $o$ by some given $x\in \Z$. 
As we have just set,
$$
\E(|\{x\in [-N,N], \, x \in I(e) \text{ for some }I(e) \text{ with } D\leq |I(e)|<D'\}|)/(2N+1)>1-c'/32.
$$
Let $\cale _N$ be the collection of all edges $e=\{k,\ell\}\in E(G)$ of length at least $D$, such that $k,\ell\in [-N,N]$. Then,
$$
\E(|\{x\in [-N,N], x\in I \text{ for some } I\in \cale_N)
\}|)/(2N+1)\geq
$$
$$
\E(|\{x\in [-N+D',N-D'], \, x\in I(e) \text{ for some }I(e) \text{ with } D\leq |I(e)|<D'\}|)/(2N+1)\geq
$$
$$
\E(|\{x\in [-N,N], \, x\in I(e) \text{ for some }I(e) \text{ with } D\leq |I(e)|<D'\}|)/(2N+1)-2D'/(2N+1)\geq 1-c'/16.
$$
Let $\bfs$ be $\{x\in [-N,N]: \text{ for every } I\in \cale _N, x\not\in I
\}$.
The previous inequalities directly imply
\begin{equation}
\E(|\bfs|)\leq c'(2N+1)/16.
\label{kimarado}
\end{equation}

\def\comp{{\bf I}}

From now on, $\comp$ will denote an arbitrary connected component of $[-N,N]\setminus\bfs$. Let $\cale (\comp)$ be the subset of edges in $\cale _N$ both of whose endpoints are in $\comp$. Now, one can
apply Lemma \ref{todelta}, for $\comp=[a,b]$, with $\{I(e):\, e\in\cale (\comp)\}$ as $\cali$, $D=\Delta$, and $d=\delta$. Let $\cali'(\comp)=\cali'$ and $\iota$ be as in Lemma \ref{todelta}. One of the implications of the lemma is that for every $k\in[-N,N]$, there is exactly 1 or 2 elements of $\cali'$ that contain $k$. From this we have
\begin{equation}
|\cali'|\leq 2|\comp|/D,
\label{eq:Delta}
\end{equation}
because every interval in $\cali'\subset\cali$ has length at least $D$.

Let $P_2\subset \comp$ be the set of those points that are contained in exactly two elements of $\cali'$, and $P_1=\comp\setminus P_2$ be the set of those that are contained in one. 
Now, let $\cals _2$ be the set of maximal connected subintervals induced by $P_2$. Consider also the set of maximal connected subintervals induced by $P_1$, and partition it into two subsets, using the natural ordering on these intervals from left to right: let $\cals_1$ be the subset of these intervals that are at odd positions at this ordering, and $\cals_3$ be the set of those that are at even positions. See Figure \ref{intervals_2}. 

\begin{figure}[h]
\vspace{0.1in}
\begin{center}
\includegraphics[keepaspectratio,scale=1.4]{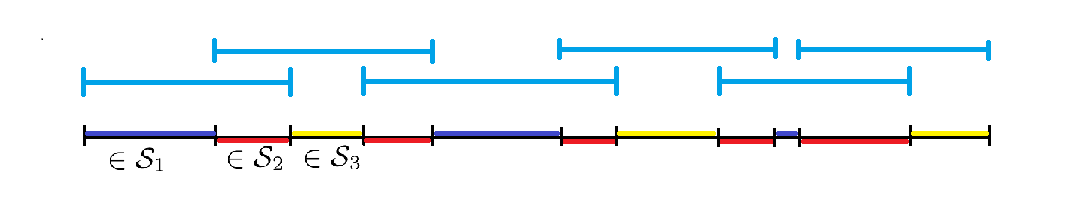}
\caption{The subinterval partition $\cali'$ of $\comp$, and the categorization of its elements to classes $\cals_1$, $\cals_2$, $\cals_3$.}\label{intervals_2}
\end{center}
\end{figure}

Denote $\cals=\cals(\comp):=\cals_1\cup\cals_2\cup\cals_3$.
We have $|\cals|\leq 2|\cali'|\leq 4|\comp|/D$ by \eqref{eq:Delta}.

If $I\subset [-N,N]$ is an arbitrary interval, let $I_-$ be its left endpoint and $I_+$ be its right endpoint. Fix $j\in\{1,2,3\}$ for now.
For every $I\in \cals_j (\comp)$, pick a path $P_I$ in $G$ of minimal length between $I_-$ and $I_+$. It is easy to check that the subgraph $\cup_{I\in \cals_j } P_I\cup \{e:\, I(e)\in \cali'\}$ of $G$ is connected, and it contains the endpoints $I_-$ and $I_+$ of $\comp$ (see Figure \ref{intervals_2}). Hence its total size is an upper bound on $\dist _G (I_-,I_+)$. We obtain that
\begin{equation}
\E (\dist _G (I_-,I_+))\leq |\cali'|+
\cup_{I\in \cals_j} |P_I| \leq
2|\comp|/D +
\E \bigl(\sum_{I\in\cals_j} \dist_G (I_-,I_+)\bigr)
\label{ezkell}
\end{equation}
using \eqref{eq:Delta}. As $I$ runs over all connected components of $[-N,N]\setminus\bfs$, one has 
\begin{equation}
\E\bigl(\sum_\comp \dist_G (\comp_-,\comp_+)\bigr)\leq \frac{1}{3}\sum_\comp\sum_{j=1}^3  \E\bigl( 2|\comp|/D+\sum_{I\in \cals_j (\comp)}\dist_G (I_-,I_+)\bigr)=
\end{equation}
\begin{equation}
(4N+2)/3D+\frac{1}{3}\Bigl(
 \E \bigl( \sum_\comp 2\delta |\cals (\comp)|+\sum_{I\in\cals (\comp)}\dist_G (\iota(I_-),\iota(I_+))\bigr)\Bigr)\leq 
\end{equation}
\begin{equation}
(4N+2)/3D+\frac{1}{3}\E \bigl(\sum_\comp 8\delta |\comp|/D\bigr)+ \frac{2}{3}\sum_{i=0}^{\lfloor N/\delta\rfloor} \E \bigl(\dist_G (i\delta, (i+1)\delta)\bigr)\leq
(4N+2)/3D+
16\delta N/3D+2 c'N/3,
\label{eq:mainbound}
\end{equation}
where the last inequality follows by unimodularity (via $\E(\dist_G (i\delta, (i+1)\delta))=\E(\dist_G (0, \delta))$) and the definition of $c'$, and the inequality before it uses Lemma \ref{todelta}.
We conclude that 
$$\E (\dist _G (-N,N))\leq 
\E(|\bfs|+\sum_\comp \dist_G (\comp_-,\comp_+))\leq 
c'(2N+1)/16+
(4N+2)/3D+
16\delta N/3D+2 c'N/3
\leq 15c'N/16.
$$
This holds for every large enough $N$, contradicting $c>15c'/16$.
\end{proof}
\qed

\begin{proofof}{Theorem \ref{oneended}}
Let $T_0$ be a fiid spanning tree of $G$ with one or two ends. Such a tree exists, as a straightforward generalization of Theorem 8.9 of \cite{AL} to the fiid setting. 

If $T_0$ has one end, then the claim is proved, so let us assume that it has 2 ends. Let $B$ be the biinfinite path in $T_0$. To every vertex $x$ in $B$, define $B_x$ as the subgraph induced in $T_0$ by $x$ and all vertices that are in a finite component of $T_0\setminus\{x\}$. For every vertex $v\in V(G)$ define $b (v)\in V(B)$ to be the (unique) vertex such that $v\in B_{b(v)}$.
We define a new unimodular graph $B^+$ on the vertex set of $B$, as a deteministic function of $(G,T_0)$.
For an edge $e=\{v,w\}$ in $G$, define $e^+=\{b(v),b(w)\}$, and let $E(B^+):=\{e^+:\, e\in E(G)\}$.
We will define an  fiid sequence $(K_n)$ of subgraphs of $B^+$ that satisfy the following:
\begin{enumerate}
\item $K_n$ is connected;
\item $\lim_{n\to\infty}\P (o\in K_n)=0$.
\end{enumerate} 
Once we have $(K_n)$, we will define a sequence $(H_n)$ of subgraphs of $G$, where $H_n:=\cup_{x\in V(K_N)} B_x \cup \{e\in E(G):\, e^+\in K_N\}$. It is easy to check that if $(K_n)$ satisfies conditions (1) and (2), then so does $(H_n)$, and thus the theorem follows from Lemma \ref{connectedsub}. It remains to construct the $K_n$.

Fix $n$ and consider Bernoulli($2^{-n})$ percolation on $V(B)$, independently from all other randomness that we have (the unimodular graph and the iid labels). For every pair of open vertices $x$ and $y$ such that every vertex of $B$ on the path between $x$ and $y$ is closed, choose a connected finite subgraph $C_{x,y}$ of minimal size of $B^+$ that contains both $x$ and $y$. Let $K_n$ be the union of all these $C_{x,y}$. Then the $K_n$ are connected. We will show that they also satisfy item 2.

As in the proof of Lemma \ref{distance}, choose a random uniform isomorphism between $B$ and $\Z$ that maps $o$ to the origin, for simpler reference. When convenient, we will refer to the vertices of $B$ as elements of $\Z$.
For an arbitrary $x\in V(B)$, let $x_+$ be the smallest $x_+>x$ that is open, and let $x_-$ be the largest $x_-\leq x$ that is open. 
Let $\eps>0$ be arbitrary. Choose $M$ such that $\E (\dist _{B^+} (o,m)/m)\leq \eps/2$ for every $m\geq M$, and choose $n_0$ so that $\P(o_+\leq M)<\eps/2$ whenever $n\geq n_0$. An $M$ with this first property exists by Lemma \ref{distance}.
Define ${\cal C}=\{C_{x,x_+}: x \text{ open}\}$. We have $K_n=\cup_{C\in{\cal C}}C$.
Define the following mass transport: let $o$ send mass $\frac{1}{o_+-o_-}$ to every vertex of $C_{o_-,o_+}$. The expected mass received is $\sum_{i=1}^\infty i\P(o \text{ is in exactly } i \text{ elements of } {\cal C})\geq \P (o \text{ is in some element of }{\cal C})=\P(o\in K_n)$. The expected mass sent out is $\E(|C_{o_-,o_+}|/|o_+-o_-|)\leq 
2\E(|\dist_{B^+}(o,o_+)|/|o_+-o|)=2\sum_{j=1}^\infty \P(o_+=j)\E(\dist _{B^+} (o,j)/j)$, using the independence of the percolation process. The first $M$ terms of this sum are less than $\eps/2$, while the sum 
$\sum_{j=M+1}^\infty \P(o_+=j)\E(\dist _{B^+} (o,j)/j)$ is also bounded by $\eps/2$. We obtain that $\P(o\in K_n)<\eps$, as we wanted. 
\end{proofof}
\qed


\end{document}